\documentclass[11pt]{article}
\usepackage{a4wide}
\usepackage{amsmath, amsthm, xcolor}
\usepackage{amssymb}
\usepackage{graphicx,soul}
\usepackage[normalem]{ulem}
\usepackage{relsize}

\newtheorem{theorem}{Theorem}
\newtheorem{proposition}{Proposition}

\newtheorem{KT}{Theorem K \!\!\!\!}

\newtheorem{KJ}{Theorem K-J \!\!\!\!}

\newtheorem{MTP}{Theorem MTP \!\!\!\!}

\newtheorem{lemma}{Lemma}

\newtheorem{lemmaCBC}{Lemma CBC (Convergence Borel--Cantelli) \!\!\!\!}

\newtheorem{lemmaDBC}{Lemma DBC (Divergence Borel--Cantelli)  \!\!\!\!}

\newtheorem{lemmaDBClocal}{Lemma LBC (Local  Borel--Cantelli)  \!\!\!\!}

\theoremstyle{remark}
\newtheorem{rem}{{Remark}}
\renewcommand{\Bbb}[1]{\mathbb{#1}}
\newcommand{\I}{[0,1]}      %
\newcommand{\N}{{\Bbb N}}         % natural numbers
\newcommand{\R}{{\Bbb R}}         % real numbers
\newcommand{\Z}{{\Bbb Z}}         % integer numbers

\newcommand{\kgb}{\cK_{G,B}}
\newcommand{\egb}{E_{G,B}}

\newcommand{\cA}{{\cal A}}

\newcommand{\cF}{{\cal F}}
\newcommand{\cG}{{\cal G}}
\newcommand{\cH}{{\cal H}}

\newcommand{\cK}{{\cal K}}

\newcommand{\ve}{\varepsilon}
\newcommand{\supp}{\operatorname{supp}}

\newcommand{\diam}{r}

\begin{document}

\large

\title{\bf The Divergence Borel--Cantelli Lemma revisited}

\author{Victor Beresnevich \and Sanju Velani}

\date{}

\maketitle

%\centerline{{\it Dedicated to ??????}}

\bigskip

\begin{abstract}
Let $(\Omega, \cA, \mu)$ be a probability space.  The classical Borel--Cantelli Lemma states that for any sequence of $\mu$-measurable sets $E_i$ ($i=1,2,3,\dots$), if the sum of their measures converges then the corresponding $\limsup$ set $E_\infty$ is of measure zero. In general the converse statement is  false. However, it is well known that the divergence counterpart is true under various additional `independence'  hypotheses. In this paper we revisit these hypotheses and establish both sufficient and necessary conditions for $E_\infty$ to have either positive or full measure.
\end{abstract}

{\footnotesize
\noindent\emph{Key words and phrases}:  $\limsup$  sets,  quasi-independence on average, Borel--Cantelli Lemma.

\noindent\emph{Mathematics Subject Classification 2000}: Primary 60F20; Secondary
11J83. }

%VB's draft

\section{Introduction}

The Borel--Cantelli Lemma is a result in probability theory with wide reaching applications to various areas of mathematics.
To some extent, this note is motivated by its deep applications to number theory, in particular to metric number theory  -- see for example  \cite{BBDV09, BDV06, durham, HarmanMNT, HarmanPap, Sprindzuk} and references within. Loosely speaking, metric number theory  is concerned with the arithmetic  properties of almost all numbers and many key results in the theory are underpinned by variants of the divergence part of the Borel--Cantelli Lemma (see Lemma~DBC below). The divergence part is also known as the second Borel--Cantelli Lemma and it naturally shows up (in some form) in the proof of the notorious Duffin-Schaeffer Conjecture \cite{DuffinSchaeffer} recently  given by Koukoulopoulos $\&$ Maynard \cite{maynard} and its higher dimensional generalisation proved two decades earlier  by Pollington $\&$ Vaughan \cite{PV90}.
Indeed, the divergence Borel--Cantelli Lemma is very much at the heart of numerous other recent advances on topical problems in metric number theory,  such as those in the theory of  multiplicative and inhomogeneous Diophantine approximation  and Diophantine approximation on manifolds and more generally on fractals, see for example
\cite{ACY, BHV20, Sam-Duke, Sam-Nicolas, Sam-Nicolas2, Sam-Lei, KL, Felipe17, Felipe2020, Yu}.
In a nutshell,  our goal it is to revisit the Borel--Cantelli Lemma and to establish both sufficient and necessary conditions that guarantee either positive or full measure.

\subsection{Background and Motivation}
 To set the scene, let $(\Omega,\cA,\mu)$ be a probability space  and
let $E_i$ ($i \in \mathbb{N}$) be a family of  measurable subsets (events) of $\Omega$.  Also, let
$$
E_{\infty}:=\limsup_{i \to \infty} E_i := \bigcap_{t=1}^{\infty}
\bigcup_{i=t}^{\infty} E_i \ ;
$$
i.e. $ E_{\infty} $ is the set of $x \in \Omega$ such that
$ x \in E_i $ for infinitely many $i \in \N$.

Determining the measure  of $E_{\infty}$ turns out to be one of the fundamental problems considered within the framework of classical probability theory -- see for example   \cite[Chp.1 \S4]{bill}   and    \cite[Chp.47]{Port}                                                                                                                      for general background and further details. With this in mind, the following \emph{convergence  Borel--Cantelli Lemma}\/ provides a beautiful and truly simple criterion for zero measure.

\bigskip

\begin{lemmaCBC}
Let $(\Omega,\cA,\mu)$ be a probability space  and let $\{E_i\}_{i \in \N} $ be a sequence of subsets (events) in $\cA$.  Suppose  that $  \sum_{i=1}^{\infty} \mu(E_i) < \infty$.
Then,
 $$\mu(E_{\infty}) = 0 \, . $$
\end{lemmaCBC}

\bigskip

\noindent  This powerful lemma, which is also known as the \emph{first  Borel--Cantelli Lemma},  has applications in numerous disciplines. In particular, within the context of number theory it is very much at the heart of Borel's proof that almost all numbers are normal \cite{borel}.

In view of  Lemma CBC, it is natural to ask whether or not there is a sufficient condition that enables us to deduce that the measure of $E_{\infty}$ is positive or possibly even full; that is to say that
$$\mu(E_{\infty}) = \mu(\Omega)=1   \, . $$

\noindent  The divergence of the measure sum $\sum_{i=1}^{\infty} \mu(E_i)$ is clearly necessary but certainly not enough as the following simple example demonstrates.

\medskip

\noindent {\em Example.}
For $i \in \N$, let  $E_i=(0, \frac1i) \subset \Omega := [0,1]$ and $ \mu$ be one-dimensional Lebesgue measure restricted to $[0,1]$. Then
$$
\textstyle{\sum_{i=1}^\infty \mu(E_i)=\sum_{i=1}^\infty i^{-1} =\infty } $$
but
$$ E_{\infty}=\bigcap_{t=1}^\infty \bigcup_{i=t}^\infty E_i  =\bigcap_{t=1}^\infty (0,\tfrac1t)=\varnothing  \qquad {\rm and \ so  \ } \quad  \mu(E_{\infty}) = 0 \, .
$$

\medskip

The problem in the above example is that the building blocks $E_i$ of the $\limsup$ set under consideration  overlap `too much' - in fact they are nested.     The upshot is that in order to have  $\mu(E_{\infty}) > 0$,   we not only need the sum  of the measures to diverge but also  that the sets $E_i$ are in `some sense' independent; that is, we need to control overlaps!
Indeed, Borel $\&$ Cantelli showed that  \emph{mutual  independence} in the classical probabilistic  sense, which means that for every $n \in \N$
\begin{equation}\label{tottalyfullind}
\mu \left( \bigcap_{t=1}^n E_{i_t}  \right)  =  \prod_{t=1}^n\mu(E_{i_t})\quad    \text{for any indices }i_1<\ldots < i_n \, ,
\end{equation}
implies that $\mu( E_{\infty}) = 1 $.
This   full measure statement, often referred to as the \emph{second Borel--Cantelli Lemma}, led to a flurry of activity with the aim of relaxing the  mutual independence condition.  Notable  progress in this quest included replacing mutual independence by {\em pairwise independence} -- this corresponds to \eqref{tottalyfullind} being fulfilled with
$n=2$ rather than every $n \in \N$.  In turn, pairwise independence was  replaced by the upper bound  condition
\begin{equation}\label{fullind}
\mu ( E_s \cap E_t  )  \le   \mu(E_s) \mu(E_t)   \qquad    s \neq t  \,
\end{equation}
on the overlaps.  Undoubtedly, verifying  \eqref{fullind} is significantly easier than \eqref{tottalyfullind}. However, in many applications, we rarely have \eqref{fullind} let alone mutual independence as in the original statement of  the second Borel--Cantelli Lemma.  What is much more useful is the following variant which these days is often referred to as the \emph{divergence Borel--Cantelli Lemma}.

\begin{lemmaDBC}
 \! Let $(\Omega,\cA,\mu)$ be a probability space  and let $\{E_i\}_{i \in \N} $ be a sequence of subsets (events) in $\cA$.
Suppose that $ \sum_{i=1}^{\infty} \mu(E_i) =\infty$
and that there exists a constant $C>0$ such that
\begin{equation}\label{vbx1x}
\sum_{s,t=1}^Q  \mu(E_s\cap E_t)\le C\left(\sum_{s=1}^Q  \mu(E_s)\right)^2
\end{equation}
holds for infinitely many $Q\in\N$.
Then
$$
\mu(E_{\infty}) \ge C^{-1}\,.
$$
In particular, if $C=1$ then $$\mu(E_{\infty}) =1 \, . $$
\end{lemmaDBC}

\medskip

We refer the reader to  \cite{HarmanMNT, HarmanPap, Port, Sprindzuk} for the proof of the lemma which is essentially a consequence of the Cauchy-Schwarz inequality. As pointed out by  Harman \cite{HarmanPap}, the basic idea goes back to the works of Payley $\&$ Zygmund \cite{PZ1+2,PZ3} from the nineteen thirties.

\bigskip
\begin{rem} \label{remrem} To the best of our knowledge, the in particular part of Lemma~DBC  first explicitly appears in the work of  Erd\"{o}s $\&$ Reyni \cite[Lemma~C]{ER} from the late fifties.   Lamperti \cite{Lamperti-63:MR0143258} in the early sixties  established the weaker form of  above lemma in which \eqref{vbx1x} is replaced by
\begin{equation}\label{pairwiseSV}
\mu ( E_s \cap E_t  )  \le  C  \mu(E_s) \mu(E_t)   \qquad    s \neq t  \, .
\end{equation}
Clearly, this pairwise condition  implies the average condition \eqref{vbx1x}. Independently and around the same time, Kochen $\&$ Stone \cite{KS} established Lemma~DBC as stated --  see also \cite[Lemma~2, p.165]{SprindzukM}. It is worth mentioning, that Chung $\&$ Erd\"{o}s \cite{Chung-Erdos-52:MR0045327}  in the early  fifties explored the implications of imposing condition \eqref{pairwiseSV} on the overlaps.
Within the specific number theoretic setting,  Duffin $\&$ Schaeffer \cite{DuffinSchaeffer} had carried out such an investigation in 1941 and it was a key ingredient in their proof  of what today is refereed to as the Duffin-Schaeffer Theorem \cite[Theorem I]{DuffinSchaeffer}. This theorem is a special case of the notorious Duffin-Schaeffer Conjecture mentioned right  at the start of this paper.
\end{rem}

\medskip

\begin{rem}  \label{0-1law}
Condition \eqref{vbx1x} is often refereed to as \emph{quasi-independence on average} and  together with the divergence of the measure sum guarantees that the associated  $\limsup$  set $E_{\infty}$ is of positive measure. It does not in general guarantee full measure. However, this is not an issue if we already know by some other means (such as Kolmogorov's theorem \cite[Theorems 4.5 \& 22.3]{bill} or ergodicity \cite[\S24]{bill}) that the $\limsup $ set  $E_{\infty}$ satisfies a \emph{zero-one law}; namely  that
\begin{equation*}
\mu(E_{\infty}) =   \   0 \quad\text{or} \quad 1.
\end{equation*}
Within the context of metric number theory, the existence of such a law for the $\limsup$ set of well approximable real numbers is due to Cassels \cite{Cassels-50:MR0036787} and Gallagher \cite{Gallagher-61:MR0133297} and it plays a key role in the recent proof of the Duffin-Schaeffer Conjecture \cite{maynard}.  For further details and higher dimensional generalisation of their zero-one laws see \cite{BHV13, BV08} and references within.
Alternatively, without the presence of a general zero-one law,  if we are willing to impose a little more structure on the probability space,   we can guarantee full measure if the measure sum diverges locally and  quasi-independence on average holds locally in the presence of an appropriate topological structure on $\Omega$.   In short, by locally we mean that the conditions under consideration  hold for $E_i  \cap A $ where $A$  is an arbitrary open set with positive measure. For the precise  statement see Lemma~LBC below.

\end{rem}

\medskip

In short, the purpose of the present paper is to determine whether or not Lemma~DBC  is best possible. In other words, is it the case that the \emph{pairwise quasi-independence on average} condition \eqref{vbx1x}  cannot be replaced by a weaker condition?    Recall,  that in view of Lemma~CBC, the divergence sum condition within Lemma~DBC is not negotiable -- it has to be present. We show that within a reasonably general framework,  given any  $\limsup$ set $E_{\infty}:=\limsup_{i \to \infty} E_i $ with $ \mu(E_{\infty}) > 0 $ the sets $E_i$ can be appropriately manipulated or rather ``trimmed''  in such a manner that the resulting  subsets $ E^*_i$ are quasi-independent on average and the sum of the measures  $  \mu(E^*_i) $ diverges. Thus, up to ``trimming'' the divergence Borel--Cantelli Lemma is best possible. Moreover, we show that quasi-independence on average for the trimmed sets is in fact not only equivalent to full measure  but to  three other useful  properties which are of independent interest especially within the context of applications.
We conclude the paper with a couple of examples that demonstrate the versatility and power of our results.

\subsection{Statement of results}

Throughout,  $(\Omega, \cA, \mu, d) $ will be a metric measure space equipped with a Borel  probability measure $\mu$. In what follows,  {$\supp \mu $ will denote the support of the measure $\mu$}
and   given $x\in\Omega$ and $r>0$, $B=B(x,r)$ will denote the ball centred at $x$ of radius $r$. {Also, given a real number $a > 0$, we denote by  $a
 B$  the ball $B$ scaled by a factor $a$; i.e.  $aB:= B(x, a r)$.}   Most of the time we will  assume that $\mu$ is doubling. Recall, that $\mu $ is said to be \emph{doubling}  if there are constants $ \lambda  \geq 1$ {and $r_0>0$} such that for any $x \in \supp\mu$ and {$0<r<r_0$}
\begin{equation}\label{doub}
\mu (B(x, 2r))  \, \leq \, \lambda \ \mu(B(x,r)) \, .
\end{equation}

\noindent The doubling condition allows us to blow up a given ball by a constant factor without drastically affecting its measure. The metric measure space $(\Omega, \cA, \mu, d) $ is also said to be  {\em doubling}\/ if $\mu$
is doubling \cite{jh}. Note that the doubling property is imposed only on the measure of balls centred in $\supp\mu$. However, in many instances the doubling property can be effectively used on balls that are not necessarily centred in $\supp\mu$ provided that they contain `enough' of the support.
In particular, when working with a given sequence of balls $ \{B_i\} $ in $\Omega$ we will often impose the following weaker version of doubling:
\begin{equation}\label{vb8}
\exists\; a,b>1\quad  {\rm such \ that} \quad \mu(aB_i)\le b\mu(B_i)\qquad \text{for all $i$ sufficiently large}.
\end{equation}
Condition \eqref{vb8}  is not particularly restrictive and  ensures that whenever $\mu(B_i)>0$ the support of $\mu$ within $B_i$ is not concentrated too close to the boundary of $B_i$.  Indeed,  if  for some $\ve>0$ the ball $(1-\ve)B_i$ contains points in $\supp \mu $, then the inequality in \eqref{vb8} holds with any $a>1$ and $b=\lambda^k$ where $k:=\left\lceil\log_2\frac{a+1-\ve}{\ve}\right\rceil$. Note that if the centre of $B_i$ is in $\supp \mu $ then the inequality in \eqref{vb8}  trivially holds with $a=2$ and $b=\lambda$.

Restricting our attention to $\limsup$ sets of balls, we have the following `if and only if'  statement for full measure.

\begin{theorem} \label{iffballs}
Let $(\Omega, \cA, \mu, d) $ be a metric measure space equipped with a doubling Borel  probability measure $\mu$. Let $\{B_i\}_{i\in\N}$ be a sequence of balls in $\Omega$ with
$\diam(B_i)\to 0$ as $i\to\infty$ and such that \eqref{vb8} holds.
Let $E_{\infty}:=\limsup_{i \to \infty} B_i  \,   $.
Then
$$
\mu(E_{\infty}) =1\,
$$
if and only if there exists a constant $C>0$ such that for any ball $B$ centred in $\supp\mu$ there is a  sub-sequence $\{L_{i,B}\}_{i\in\N} $ of  $\{B_i\}_{i\in\N}$  of balls contained in $B$ $($i.e. $L_{i,B}\subset B$ for all $i)$, such that
\begin{equation}\label{onemore1}
\sum_{i=1}^{\infty} \mu(L_{i,B}) =\infty
\end{equation}
and for infinitely many $Q \in \N$
    \begin{equation} \label{onemore2}
     \sum_{s,t=1}^Q  \mu\big(L_{s,B}\cap L_{t,B} \big) \ \le \  \frac{C}{\mu(B)}  \,  \left(\sum_{s=1}^Q  \mu(L_{s,B})\right)^2 \ \, .
   \end{equation}
\end{theorem}

\bigskip

It is important to note that the constant $ C>0$ appearing in \eqref{onemore2} is independent of the arbitrary ball $B$.
The following is a  strengthening of Theorem~\ref{iffballs} to $\limsup$ sets of open sets.  As we shall see, the proof will follow the same line of argument as that of Theorem~\ref{iffballs}.

\begin{theorem} \label{opensets}
Let $(\Omega, \cA, \mu, d) $ be a metric measure space equipped with a doubling Borel  probability measure $\mu$. Let $\{E_i\}_{i\in\N}$ be a sequence of open subsets (events)  in $\Omega$  and let $E_{\infty}:=\limsup_{i \to \infty} E_i  \,   $.
Then
$$
\mu(E_{\infty}) =1\,
$$
if and only if there exists a constant $C>0$ such that for any ball $B$ centred in $\supp\mu$ there is a sequence $\{L_{i,B}\}_{i\in\N}$ of finite unions of disjoint balls centred in $\supp\mu$ with $L_{i,B}\subset E_i\cap B$ satisfying \eqref{onemore1} and \eqref{onemore2} for infinitely many $Q \in \N$.
\end{theorem}

\bigskip

The upshot is that for a $\limsup$ set $E_{\infty} $ to have full measure we must  be able to locally ``trim'' the associated sets $E_i$ so that the  resulting trimmed subsets are  quasi-independent on average and the sum of their measures  diverges.

It turns outs that the sufficiency part of the  Theorem~\ref{opensets}  can be made more general. In particular, the doubling condition can be altogether dropped.
The following is a
local variant of the (standard) divergence Borel--Cantelli Lemma which allows us to deduce full measure rather than just positive measure.

\begin{lemmaDBClocal}
Let $(\Omega, \cA, \mu, d) $ be a metric measure space equipped with a Borel  probability measure $\mu$  and  let $\{E_i\}_{i \in \N} $ be a sequence of Borel subsets of $\Omega$.  Suppose there exists an increasing function $f:(0,+\infty)\to(0,+\infty)$ with $f(x)\to0$ as $x\to0$ such that for any open set $A $ with $\mu(A) > 0$
 there is  a sequence $\{L_{i,A}\}_{i\in\N} $ of measurable subsets of $A$ such that
 \begin{equation}\label{elite}
\sum_{i=1}^{\infty} \mu(L_{i,A}) =\infty
\end{equation}
and for infinitely many $Q \in \N$
    \begin{equation}\label{vbx1xslvA}
     \sum_{s,t=1}^Q  \mu\big(L_{s,A}\cap L_{t,A} \big) \ \le \  \frac{1}{f(\mu(A))}  \,  \left(\sum_{s=1}^Q  \mu(L_{s,A})\right)^2 \ \, .
\end{equation}
Then
$$
\mu(E_{\infty}) =1\,.
$$
Moreover, if in addition  $\mu$ is doubling and $f(x)=cx$ for some constant $0<c \le 1$, it suffices to take  $A$ in the above to be an arbitrary ball of sufficiently small radius centred in $\supp\mu$.
\end{lemmaDBClocal}

\medskip

\begin{rem}
In the case $f(x)=cx$ for some constant $0<c \le  1$ and $A=B$ is a ball, condition \eqref{vbx1xslvA} becomes the same as \eqref{onemore2} with $C=c^{-1}$.
Given a measurable set $A$ with $\mu(A)>0$, let $\mu_A$ denote the conditional probability measure given by
$$
  \mu_A(E):=\frac{\mu(E\cap A)}{\mu(A)} \qquad\text{for $E\in\cA$}\, .
$$
In other words,  $\mu_A$ is the re-normalised $\mu$-measure restricted to $A$.  Then it is easily seen that on replacing $\mu $ by $\mu_A$,  the divergence condition \eqref{elite} and  the overlap condition \eqref{vbx1xslvA} with $f(x)=cx$ coincide with those of Lemma~DBC. For obvious reasons, the independence condition \eqref{vbx1xslvA} with $f(x)=cx$
is often refereed to as \emph{local quasi-independence on average}.
\end{rem}

\bigskip

Theorem~\ref{iffballs} will follow from a more general statement that provides three more necessary and sufficient conditions for full measure. To be more precise, Theorem~\ref{iffballs} is the equivalence between (A) and (E) within the following statement with $C=\kappa^{-2}$.

\noindent Throughout,  we use the standard notation $\bigcup\limits^\circ$ to denote that the union of sets under consideration is disjoint.

\begin{proposition} \label{sv2021}
Let $(\Omega, \cA, \mu, d) $ be a metric measure space equipped with a doubling Borel  probability measure $\mu$. Let $\{B_i\}_{i\in\N}$ be a sequence of balls in $\Omega$ with
$\diam(B_i)\to 0$ as $i\to\infty$ and such that \eqref{vb8} holds.  Let $E_{\infty}:=\limsup_{i \to \infty} B_i  \,   $. Then, the following statements are equivalent:

\begin{itemize}
  \item[\rm \textbf{(A)}\ ]   $\mu (E_{\infty}) =1 $.
  \item[\rm \textbf{(B)}\ ]   For any ball
$B$  in $\Omega$, we have that
\begin{equation}\label{e:011}
\mu (E_{\infty} \cap B) =  \mu(B) \, .
\end{equation}
 \item[\rm \textbf{(C)}\ ]   For any ball $B$ in $\Omega$ centred in $\supp\mu$ and any $G \in \N$, there is a finite
sub-collection  $\kgb\subset\{B_i\,:\,i\ge G\}$ of disjoint balls contained in $B$  such that
\begin{equation}\label{e:014}
\mu\Big(\bigcup^\circ_{L\in \kgb}\!\!L\Big) \ \ge \ \kappa
 \ \mu(B) \,
 \hspace{8mm}\text{where}\qquad \kappa :=  \frac12 \;\frac{1}{\lambda^{k+1}b} \ \,
\end{equation}
where $\lambda$ is as in \eqref{doub}, $k:=\max\{1,\left\lceil\log_2\frac{6}{a-1}\right\rceil\}$ and $a,b$ are as in \eqref{vb8}.
\item[\rm \textbf{(D)}\ ]   For any ball $B$ in $\Omega$ centred in $\supp\mu$  and any $G\in \N$, there is a subset $ \egb \subseteq B$ consisting of a finite union of disjoint balls from  $\{B_i:i\ge G\}$,  such that
$$  \sum_{i=1}^{\infty} \mu(E_{G_i,B}) =\infty  \, $$
for any subsequence $(G_i)_{i\in\N}$ of natural numbers, and, with $\kappa$ is as in \eqref{e:014}, for any pair of natural numbers $G$ and $G'$
   \begin{equation}\label{snow1}
     \mu\big(\egb \cap E_{G',B} \big) \ \le \  \frac{1}{\mu(B) \; \kappa^2}  \ \ \mu(\egb)  \, \mu(E_{G',B})  \, .
   \end{equation}

\item[\rm \textbf{(E)}\ ]   For any ball $B$ in $\Omega$ centred in $\supp\mu$ there is a  sub-sequence  $\{L_{i,B}\}_{i\in\N} $ of  $\{B_i\}_{i\in\N}$ of balls contained in $B$  such that
$$  \sum_{i=1}^{\infty} \mu(L_{i,B}) =\infty  \, $$
and, with $\kappa$ is as in \eqref{e:014}, for infinitely many $Q \in \N$
   \begin{equation}\label{snow2}
     \sum_{s,t=1}^Q  \mu\big(L_{s,B}\cap L_{t,B} \big) \ \le \  \frac{1}{\mu(B) \; \kappa^2}  \,  \left(\sum_{s=1}^Q  \mu(L_{s,B})\right)^2  \, .
    \end{equation}

\end{itemize}

\end{proposition}

\bigskip

\begin{rem} It will become apparent in the proof that we can take the subset $ \egb $ in (D) to be the union of balls in the sub-collection $\kgb$ associated with (C).
\end{rem}

\medskip

We now turn our attention to an `if and only if' statement for positive measure  for $\limsup$ sets of balls.

\begin{theorem} \label{iffballspositive}
Let $(\Omega, \cA, \mu, d) $ be a metric measure space equipped with a doubling Borel  probability measure $\mu$. Let $\{B_i\}_{i\in\N}$ be a sequence of balls in $\Omega$
%with $\diam(B_i)\to 0$ as $i\to\infty$ and
such that \eqref{vb8} holds.
Let $E_{\infty}:=\limsup_{i \to \infty} B_i  \,   $.
Then
$$
\mu(E_{\infty}) > 0 \,
$$
if and only if there exists a  sub-sequence $\{L_{i}\}_{i\in\N} $ of  $\{B_i\}_{i\in\N}$ and
a constant $C>0$  such that
\begin{equation}\label{vb30}
\sum_{i=1}^{\infty} \mu(L_{i}) =\infty
\end{equation}
and for infinitely many $Q \in \N$
\begin{equation}\label{vb31}
     \sum_{s,t=1}^Q  \mu\big(L_{s}\cap L_{t} \big) \ \le \  C \,  \left(\sum_{s=1}^Q  \mu(L_{s})\right)^2 \ \, .
\end{equation}
\end{theorem}

\bigskip

The following is the `positive measure' analogue of Proposition~\ref{sv2021} and it clearly implies Theorem~\ref{iffballspositive}.

\begin{proposition} \label{sv2021positive}
Let $(\Omega, \cA, \mu, d) $ be a metric measure space equipped with a doubling Borel  probability measure $\mu$. Let $\{B_i\}_{i\in\N}$ be a sequence of balls in $\Omega$
%with $\diam(B_i)\to 0$ as $i\to\infty$
such that \eqref{vb8} holds. Let $E_{\infty}:=\limsup_{i \to \infty} B_i  \,   $. Then, the following statements are equivalent:

\begin{itemize}
  \item[\rm \textbf{(A)}\ ]   $\mu (E_{\infty}) >0 $.
 \item[\rm \textbf{(B)}\ ]   For any $G \in \N$, there is a finite
sub-collection  $\cK_G\subset\{B_i\,:\,i\ge G\}$ of disjoint balls  such that
\begin{equation}\label{e:014snow}
\mu\Big(\bigcup^\circ_{L\in \cK_G}\!\!L\Big) \ \ge \ \kappa
  \,
 \hspace{8mm}\text{where}\qquad  \kappa :=  \frac{1}{2\lambda^{k+1}b}  \ \mu \big(E_{\infty}\big) \,,
\end{equation}
where $\lambda$ is as in \eqref{doub}, $k:=\max\{1,\left\lceil\log_2\frac{6}{a-1}\right\rceil\}$ and $a,b$ are as in \eqref{vb8}.

\item[\rm \textbf{(C)}\ ]   For any  $G\in \N$, there is a subset $ E_G$ of $\Omega$, which is a finite union of disjoint balls from  $\{B_i:i\ge G\}$,  such that
$$  \sum_{i=1}^{\infty} \mu(E_{G_i}) =\infty  \, $$
for any subsequence $(G_i)_{i\in\N}$ of natural numbers, and, with $\kappa$ is as in \eqref{e:014snow}, for any pair of natural numbers $G$ and $G'$
   \begin{equation}\label{snow1sv}
     \mu\big(E_G \cap E_{G'} \big) \ \le \  \kappa^{-2}  \ \ \mu(E_G)  \, \mu(E_{G'})  \, .
   \end{equation}

\item[\rm \textbf{(D)}\ ]  There is a  sub-sequence  $\{L_{i}\}_{i\in\N} $ of  $\{B_i\}_{i\in\N}$   such that
$$  \sum_{i=1}^{\infty} \mu(L_{i}) =\infty  \, $$
and, with $\kappa$ is as in \eqref{e:014snow}, for infinitely many $Q \in \N$
   \begin{equation}\label{snow2sv}
     \sum_{s,t=1}^Q  \mu\big(L_{s}\cap L_{t} \big) \ \le \   \kappa^{-2}  \,  \left(\sum_{s=1}^Q  \mu(L_{s})\right)^2 \,.
    \end{equation}

\end{itemize}

\end{proposition}

\section{Proof of results \label{proofs}}

\subsection{Preliminaries}

We will make multiple use of the
 following basic covering lemma, see for example   \cite{jh, mat}.

\begin{lemma}[The $5r$ covering lemma]\label{5r}
Every family ${\cal F}$ of balls of uniformly bounded diameter in
a metric space $(\Omega,d)$ contains a disjoint subfamily ${\cal G} $
such that $$ \bigcup_{B \in {\cal F} } B \ \subset \ \bigcup_{B
\in {\cal G} } 5B \ \ \ . $$
\end{lemma}

The  following measure theoretic result, which is an extension of Proposition~1 in \cite[\S8]{BDV06}, provides a mechanism for establishing full measure statements.

\begin{lemma}\label{snowya}
Let $(\Omega, \cA, \mu, d) $ be a metric measure space equipped with a Borel doubling  probability measure $\mu$. Let $E$
be a Borel subset of\/ $\Omega$. Assume that there are constants
$r_0,c>0$ such that for any ball $B$ centred in $\supp\mu$ with $r(B)<r_0$, we have that
\begin{equation}\label{vb7}
\mu(E\cap B)\ge c\ \mu(B).
\end{equation}
Then,
$\mu(E)  =   1  $.
\end{lemma}

The lemma is a standard corollary of the Lebesgue density theorem or more generally the Lebesgue differentiation theorem  for doubling metric measure spaces (see for example \cite[Theorem~1.8]{jh}.
A slightly weaker version of this lemma can also be found in \cite{BDV06}. In short, the version of this lemma established as Proposition~1 in \cite[\S8]{BDV06} requires that \eqref{vb7} holds for arbitrary balls centred in $\Omega$ rather than just in $\supp\mu$  and the proof uses covering arguments rather than the Lebesgue density theorem.

\begin{rem}
Note that the doubling assumption in Lemma~\ref{snowya} can be weakened by requiring instead that $(\Omega, \cA, \mu, d) $ is a Vitali space as defined in \cite[p.6]{jh}.
Furthermore it is also possible to remove the doubling assumption altogether from Lemma~\ref{snowya} at the price of requiring a lower bound on $\mu(E\cap B)$ for an arbitrary open set $B$ as opposed to an arbitrary ball of sufficiently small radius. We will state this version formally as it will be required in the proof of Lemma~LBC.
\end{rem}

\medskip

\begin{lemma}[Lemma~6 in \cite{BDV06}]\label{MEMLEM6}
Let $(\Omega, \cA, \mu, d) $ be a metric measure space equipped with a Borel probability measure $\mu$. Let $E$
be a Borel subset of\/ $\Omega$ and $f:(0,+\infty)\to(0,+\infty)$ be an increasing function such that $f(x)\to0$ as $x\to0$. Assume that
$$
\mu(E\cap A)\ge f(\mu(A))
$$
for any open subset $A\subset\Omega$ with $\mu(A)>0$. Then,
$\mu(E)  =   1  $.
\end{lemma}

\medskip

The following ``obvious'' but useful statement relates the standard doubling property \eqref{doub} for balls centred in $\supp \mu$ with the weaker property corresponding to \eqref{vb8} in which the centre can be anywhere.

\begin{lemma}\label{lem_vb}
Let $(\Omega, \cA, \mu, d) $ be a metric measure space equipped with a Borel doubling probability measure $\mu$. Let $a,b>0$ be constants and let $B$ be a ball in $\Omega$ such that $\mu(aB)\le b\mu(B)$ and $B\cap\supp\mu\neq\varnothing$. Then for any $s\ge a$ we have that
$$
\mu(sB)\le \lambda^kb \mu(B)\,,
$$
where $k\in\N$ satisfies $2^k\ge(1+s)/(a-1)$ and $\lambda $ is as in \eqref{doub}.
\end{lemma}

\begin{proof}
Let $x\in B\cap\supp\mu$ and $B'$ be the ball centred at $x$ of radius $(a-1) r(B)$. Then clearly $B'\subset aB$.
Since $B'$ is centred in the support of $\mu$, the doubling property \eqref{doub} is applicable to it and we have that for every $k\in\N$
$$
\mu(2^kB')\le \lambda^k\mu(B')\,.
$$
Since $B'\subset aB$ and $\mu(aB)\le b\mu(B)$, we therefore have that
\begin{equation}\label{eq1}
\mu(2^kB')\le \lambda^k\mu(aB)\le \lambda^kb\mu(B)\,.
\end{equation}
It remains to observe that $sB\subset 2^kB'$ provided that $(a-1)2^k\ge1+s$  and so  \eqref{eq1} implies the required inequality.
\end{proof}

\subsection{Proof of Lemma~LBC}

Let $A$ be any open subset of $\Omega$ and $\{L_{i,A}\}_{i\in\N} $ be the sequence of sets as in Lemma~LBC. In particular, by definition, $L_{i,A}\subset E_i\cap A$ for every $i\in\N$ and therefore
$$
\limsup_{i\to\infty} L_{i,A}\ \subseteq \
A\cap\limsup_{i\to\infty}E_i \, .
$$
On applying Lemma~DBC (the standard divergent Borel--Cantelli Lemma) it follows that
$$
\mu \Big( A\cap\limsup_{i\to\infty} E_i \Big)  \ \geq \  \mu \Big(
\limsup_{i\to\infty} L_{i,A}  \Big) \ \geq \  f(\mu(A))\, .
$$
Also recall that $E_{i}$ is a Borel set for every $i\in\N$ and therefore $E_{\infty}:=\limsup_{i\to\infty} E_i$ is a Borel subset of $\Omega$. Then, applying Lemma~\ref{MEMLEM6} with $E= E_{\infty}$  implies  that $ \mu (  E_\infty) = 1 $ as desired.

If $\mu$ is a doubling measure and $f(x)=cx$ for some constant $0<c \le 1$, then the  `moreover' part of Lemma~LBC follows on applying Lemma~\ref{snowya} instead of Lemma~\ref{MEMLEM6}.

%
%Details of the following  two measure theoretic results can be found in \cite[\S8]{BDV06}.
%
%
%
%
%\begin{lemma}\label{lem1a}
%
%Let $(\Omega, \cA, \mu, d) $ be a metric measure space equipped with a doubling  probability measure $\mu$ such
%that any open set is $\mu$ measurable. Let $E$
%be a Borel subset of\/ $\Omega$. Assume that there are constants
%$r_0,c>0$ such that for any ball $B$ with $r(B)<r_0$  and center
%in $\Omega$,  we have that $\mu(E\cap B)\ge c\ \mu(B)$. Then, for any
%ball $B$
%$$\mu(E\cap B) \, =  \, \mu(B)  \ . $$
%\end{lemma}
%
%
%\begin{lemma}\label{lem3a}
%Let $(\Omega, \cA, \mu, d) $ be a metric measure space equipped with a probability measure $\mu$.
%Let $B$ be a ball in $\Omega$ and let $\{E_n\}_{n \in \N} $ be  a sequence of
%$\mu$--measurable sets. Suppose there exists a constant $c > 0$
%such that $\limsup_{n \to \infty} \mu(B \, \cap \, E_n) \ge c \;
%\mu(B)$. Then $$ \mu (B \, \cap \, \limsup_{n\to\infty}E_n ) \ \ge
%\ c^2 \, \mu(B) \ \ . $$
%\end{lemma}
%
%
%

\subsection{Proof of Proposition \ref{sv2021}}

\noindent $\bullet$ Step 1:  (A) $\Longrightarrow $ (B). \   This is obvious since $\mu$ is a probability measure.

\bigskip

\noindent $\bullet$ Step 2:  (B) $\Longrightarrow $ (C). \ Let $B$ be any ball in $\Omega$ centred in $\supp\mu$. In particular, we have that $\mu(B)>0$.  Let
$\cF:=\{B_i\,:\,B_i\ \cap \tfrac12B \cap\supp\mu \neq \emptyset \,
,\ i\ge G\}$.  Since, $\diam(B_i)\to0$ as
$i\to\infty$,  we can ensure that every ball in $ \cF$ is contained
in $ B $ by increasing $G$ if necessary.
In view of the $5r$ covering lemma (Lemma \ref{5r}), there exists a disjoint sub-family ${\cal
G} $ such that
$$
\bigcup_{B_i\in {\cal F} } B_i \ \subset \ \bigcup_{B_i \in {\cal G} } 5B_i \ \ \ .
$$
It follows that
$$
\tfrac12 B \cap
\limsup_{i\to\infty}B_i \cap \supp\mu \ \subset \ \bigcup_{B_i\in {\cal F} } B_i \cap \supp\mu
\ \subset \ \bigcup_{B_i \in {\cal G} } 5B_i \cap \supp\mu\ \ \ .
$$
Hence,
\begin{eqnarray*}
\mu \left(\bigcup_{B_i \in {\cal G} } 5B_i\right)  \
\geq \ \mu \big( \mbox{$\frac{1}{2}$} B \cap
\limsup_{i\to\infty}B_i \big)  \stackrel{(\ref{e:011})}{\ = \ }
  \mu(\mbox{$\frac{1}{2}$} B \big) \stackrel{(\ref{doub})}{\ \geq \ }
\frac{1}{\lambda}   \; \mu(B) \ .
\end{eqnarray*}

\noindent However, since ${\cal G} $ is a disjoint collection of
balls, which have non-empty intersection with $\supp\mu$, we have that
\begin{eqnarray*}
\mu \left(\bigcup_{B_i \in {\cal G} } 5B_i\right) \ \leq \ \sum_{B_i \in {\cal G} } \mu \left( 5B_i\right)
\stackrel{\eqref{vb8}\;\&\;\text{Lemma~\ref{lem_vb}}}{\ \leq \ }     \ \lambda^kb   \  \sum_{B_i \in {\cal G} } \mu \left( B_i\right) \ =\     \ \lambda^kb \  \
\mu \left(\bigcup^\circ_{B_i \in {\cal G} } B_i\right)\
\ ,
\end{eqnarray*}
where $k:=\max\{1,\left\lceil\log_2\frac{6}{a-1}\right\rceil\}$.
Thus, \begin{equation}
\frac{1}{\lambda^{k+1}b} \ \mu(B)  \; \le \; \mu \left(\bigcup^\circ_{B_i \in
{\cal G} } B_i\right)\ = \sum_{B_i \in
{\cal G} }  \mu \left(B_i\right) \ \leq
\ \mu(B) \ . \label{kgb001} \end{equation}
If $\cG$ is infinite, the sum in \eqref{kgb001} is convergent and therefore
there exists some $j_0 > G $ for which
\begin{equation}
\sum_{B_i \in {\cal G} \, : \, i
\geq j_0  }\mu \left( B_i\right)\ = \ \mu \left(\bigcup^\circ_{B_i \in {\cal G} \, : \, i
\geq j_0  } B_i\right) \ < \ \frac12 \;\frac{1}{\lambda^{k+1}b}
 \; \ \mu(B) \ .
\label{kgb002}
\end{equation}
Obviously,  this is also true if $\cG$ is finite.
Now let $\kgb := \{B_i : B_i \in
{\cal G} , i < j_0 \} $. Clearly, this is a finite sub-collection
of $\{B_i\,:\,i\ge G\}$. Moreover, in view of (\ref{kgb001}) and
(\ref{kgb002}) the collection $\kgb$
satisfies the desired properties.

 \bigskip

\noindent $\bullet$ Step 3:  (C) $\Longrightarrow $ (D). \   For any ball
$B$ centred in $\supp\mu$ and any $G\in \N$,
 let $\kgb$ be the finite sub-collection of disjoint balls associated with (C) and define
\begin{equation}\label{vb22}
E_{G,B} \, :=  \bigcup_{L\in\kgb}^{\circ} \!\! L \  \ \subseteq \ B \, .
\end{equation}
It follows from (\ref{e:014}) that $$\mu(E_{G,B})  \ge  \kappa \,
\mu(B) \, ,  $$

\noindent which in turn implies that
$  \sum_{i=1}^{\infty} \mu(E_{G_i,B}) \geq   \sum_{i=1}^{\infty}  \kappa \,
\mu(B)   =\infty  \, $ for any subsequence $(G_i)_{i\in\N}$ of natural numbers, and
$$
     \mu\big(\egb \cap E_{G',B} \big) \ \le \  \mu(B) \ \le \  \frac{1}{\mu(B) \; \kappa^2}  \ \ \mu(\egb)  \, \mu(E_{G',B})  \,
$$
for any  pair of natural numbers $G$ and $G'$.
Thus, the sets $E_{G,B}$ satisfy the desired properties.

 \bigskip

\noindent $\bullet$ Step 4: (D) $\Longrightarrow $ (E). \ Let $B$ be any ball
centred in $\supp\mu$ and for any $G\in \N$ let $E_{G,B}\subset B$ be as in (D), and let $\kgb$ be a finite collection of disjoint balls from $\{B_i:i\ge G\}$ that constitute $E_{G,B}$, that is \eqref{vb22} holds. Observe that for any pair of natural numbers $G$ and $G'$
   \begin{eqnarray}  \label{snow3}
   \mu\big(\egb \cap E_{G',B} \big) \ &  =   &  \  \mu\left(  \Big( \bigcup^\circ_{L\in \kgb}\!\!L     \Big) \ \cap  \ \Big(\bigcup^\circ_{L'\in \cK_{G',B}}\!\!L'\Big) \right) \nonumber \\[0ex]
   & = & \
    \sum_{L \in \kgb}  \ \sum_{L' \in \cK_{G',B} }      \mu \Big( L \cap \ L'  \Big)    \nonumber  \\[0ex]
    &\stackrel{(\ref{snow1})}{\ \leq \ }  &  \ \frac{1}{\mu(B) \; \kappa^2}  \ \ \mu(\egb)  \  \  \mu(E_{G',B})  \nonumber  \\[0ex]
    & = & \ \frac{1}{\mu(B) \; \kappa^2}  \ \ \sum_{L \in \kgb}  \mu(L)  \ \sum_{L' \in \cK_{G',B} }  \mu(L')   \, .
   \end{eqnarray}
Let $G_1=1$ and fix the collection $\cK_{G_1,B}$. Define $G_2=t+1$ where $t$ is the largest index such that $B_t\in \cK_{G_1,B}$. Since $\cK_{G_1,B}$ is finite this is clearly possible. With $G_2$ defined, we can fix the collection $\cK_{G_2,B}$ and proceed by induction as follows. Suppose the integers $G_1,\dots,G_n$ and the corresponding collections $\cK_{G_1,B}$,\dots,$\cK_{G_n,B}$ have been determined. Define $G_{n+1}=t+1$  where $t$ is the largest index such that $B_t\in \cK_{G_n,B}$. With $G_{n+1}$ defined,  we  can fix the collection $\cK_{G_{n+1},B}$ and we are done. Now, let $\{L_{s, B} \}_{s \in \N} $ be the sequence of balls contained in $B$ obtained by placing the balls from $\{\cK_{G_i,B}: i \in \N\}$ in the same order as in $\{B_i\}_{i\in\N}$. In view of the choice of the integers $G_i$, the sequence  $\{L_{s, B} \}_{s \in \N} $ is a well defined sub-sequence of $\{B_i\}_{i\in\N}$. For $M\in \N$, let
  $$
   Q_M \  := \ \sum_{i=1}^M   \# \cK_{G_i,B}  \, .
 $$
It then follows that for any $M \geq 2 $
 \begin{eqnarray*}  \label{snow4}
  \sum_{s,t=1}^{Q_M}  \mu\big(L_{s,B}\cap L_{t,B} \big)  \ &  =   &  \    \sum_{i,j=1}^{M} \mu\big(E_{G_i,B} \cap E_{G_j,B} \big) \nonumber \\[0ex]
   & = & \  \sum_{i,j=1}^{M}  \
    \sum_{L \in \cK_{G_i,B}}  \ \sum_{L' \in \cK_{G_j,B} }      \mu \Big( L \cap \ L'  \Big)    \nonumber  \\[0ex]
   &\stackrel{(\ref{snow3})}{\ \leq \ }  & \ \frac{1}{\mu(B) \; \kappa^2}  \ \  \sum_{i,j=1}^{M}  \ \sum_{L \in \cK_{G_i,B}}  \mu(L)  \ \sum_{L' \in \cK_{G_j,B} }  \mu(L')   \nonumber  \\[0ex]
   & =  & \ \frac{1}{\mu(B) \; \kappa^2}  \ \  \Big( \sum_{i=1}^{M}  \ \sum_{L \in \cK_{G_i,B}}  \mu(L)  \Big)^2  \\[0ex]
   & =  & \ \frac{1}{\mu(B) \; \kappa^2}  \ \  \Big( \sum_{s=1}^{Q_M}   \mu(L_{s,B})  \Big)^2
     \, .
   \end{eqnarray*}
This together with the fact that
$$  \sum_{s=1}^{\infty} \mu(L_{s,B}) =  \ \sum_{i=1}^{\infty} \mu(E_{G_i,B}) =  \infty  \, , $$
shows that the sequence $\{L_{s, B} \}_{s \in \N} $  satisfies the desired properties.

\bigskip

\noindent $\bullet$ Step 5:  (E) $\Longrightarrow $ (A). \  This follows immediately on applying the ``moreover'' part of Lemma~LBC with $c=\kappa^2$.

%%%%%%%%%%%%%%%%%%%%%%%%%%%%%%%%%%%%%%%%%%%%%%%%%%%%%%

\subsection{Proof of Proposition~\ref{sv2021positive}}

The proof is very similar to that of  Proposition~\ref{sv2021} and  so we will simply provide a sketch.

\medskip

\noindent $\bullet$ Step 1:  (A) $\Longrightarrow $ (B). \ For  $G\in \N$, let
$\cF:=\{B_i\cap\supp\mu\neq\varnothing\,:\, i\ge G\}$.
In view of the $5r$ covering
lemma (Lemma \ref{5r}), there exists a disjoint sub-family ${\cal
G} $  of $\cF$ such that $$ \bigcup_{B_i\in {\cal F} } B_i \ \subset \
\bigcup_{B_i \in {\cal G} } 5B_i \ \ \ . $$ It follows that
\begin{eqnarray*}
\mu \left(\bigcup_{B_i \in {\cal G} } 5B_i\right)  \
\geq \ \mu \big(
E_{\infty} \big) \
\end{eqnarray*}
and the same argument leading to \eqref{kgb001} shows that
 \begin{equation}
 \mu \left(\bigcup^\circ_{B_i \in
{\cal G} } B_i\right) \ \geq \ \frac{1}{\lambda^{k+1}b} \
\ \mu \big(
E_{\infty} \big) \ , \label{kgb001snow}
\end{equation}
where $k$ is the same integer  as in  \eqref{kgb001}. Furthermore,  the same argument leading to \eqref{kgb002} shows that
there exists some $j_0 > G $ for which
\begin{equation}  \mu \left(\bigcup^\circ_{B_i \in {\cal G} \, : \, i
\geq j_0  } B_i\right) \ < \ \frac12\;\frac{1}{\lambda^{k+1}b}
 \; \ \mu \big(
E_{\infty} \big) \, := \, \kappa.
\label{kgb002snow} \end{equation}  Then, in view of (\ref{kgb001snow}) and
(\ref{kgb002snow}) the finite sub-collection $\cK_G := \{B_i : B_i \in
{\cal G} , i < j_0 \} $ of $\{B_i\,:\,i\ge G\}$
satisfies the desired properties.

 \bigskip

\noindent $\bullet$ Step 2:  (B) $\Longrightarrow $ (C). \   For any  $G\in \N$,
 let $\cK_G$ be the finite sub-collection of disjoint balls associated with (B) and define
\begin{equation}\label{vb26}
E_{G} \, :=  \bigcup_{L\in\cK_G}^{\circ} \!\! L \,.
\end{equation}
It follows from (\ref{e:014snow}) that $\mu(E_{G})  \ge  \kappa \,
  $  which in turn implies that
$  \sum_{i=1}^{\infty} \mu(E_{G_i}) =\infty  \, $ for any subsequence $(G_i)_{i\in\N}$ of natural numbers, and that
$$
     \mu\big(E_G \cap E_{G'} \big) \ \le \  1 \ \le \  \kappa^{-2}  \ \ \mu(E_G)  \, \mu(E_{G'})  \,
$$
for any  pair of natural numbers $G$ and $G'$.
Thus, the sets $E_{G}$ satisfy the desired properties.

 \bigskip

\noindent $\bullet$ Step 3:  (C) $\Longrightarrow $ (D). \
\ For any $G\in \N$ let $E_{G}\subset \Omega$ be as in (C), and let $\cK_G$ be a finite collection of disjoint balls from $\{B_i:i\ge G\}$ that constitute $E_{G}$, that is \eqref{vb26} holds. Observe that the same argument leading to \eqref{snow3} shows that for any pair of natural numbers $G$ and $G'$
   \begin{eqnarray}  \label{snow3sv}
   \mu\big(E_G \cap E_{G'} \big) \ &  =  &  \
    \sum_{L \in \cK_G}  \ \sum_{L' \in \cK_{G'} }      \mu \Big( L \cap \ L'  \Big)    \nonumber  \\[1ex]
    &\stackrel{(\ref{snow1sv})}{\ \leq \ }  & \ \kappa^{-2}  \ \ \sum_{L \in \cK_G}  \mu(L)  \ \sum_{L' \in \cK_{G'} }  \mu(L')   \, .
   \end{eqnarray}
   Now, let  $\{L_{s} \}_{s \in \N} $ be the sub-sequence of $\{B_i\}_{i\in\N}$ balls corresponding to the sub-collections $\{\cK_{G_i} : i \in\N\} $, where the sequence of natural numbers $G_1,G_2,\dots$ is defined in the same way as within Step~4 of the proof of Proposition~\ref{sv2021}. For $M\in \N$, let $
   Q_M  :=  \sum_{i=1}^M   \# \cK_{G_i}  \, .
 $
 Then, the same argument used within Step~4 of the proof of Proposition~\ref{sv2021}, shows that the sequence $\{L_{s} \}_{s \in \N} $  satisfies the desired properties.

\bigskip

\noindent $\bullet$ Step 4:  (D) $\Longrightarrow $ (A). \  By definition,
$
\limsup_{i\to\infty} L_{i}\ \subseteq \
E_{\infty}
$.
Thus, on applying Lemma~DBC (the standard divergent Borel--Cantelli Lemma) it follows that
$$
\mu \big( E_\infty \big)  \ \geq \  \mu \Big(
\limsup_{i\to\infty} L_{i}  \Big) \ \geq \  \kappa^2  \, > \, 0  \, .
$$

\subsection{Proof of Theorem~\ref{opensets}}

The sufficiency side of Theorem~\ref{opensets} is an immediate consequence of the ``moreover'' part of Lemma~LBC. Thus we only have to prove the necessity side.  This would clearly follow
on mimicking the proof of Proposition~\ref{sv2021} if we  could establish the analogue of Part~(C) from \eqref{e:011} which trivially follows from our working assumption   that $\mu(E_\infty)=1$.
Thus, with this in mind, let $B$ be any ball in $\Omega$ centred in $\supp\mu$. In particular, we have that $\mu(B)>0$.
Let $\cF:=\{B(x)\subset E_i\cap B\,:\,x\in B\cap E_i\cap\supp\mu\,
,\ i\ge G\}$.
In view of the $5r$ covering lemma (Lemma \ref{5r}), there exists a disjoint sub-family ${\cal
G} $ such that
$$
\bigcup_{L\in {\cal F} } L \ \subset \ \bigcup_{L \in {\cal G} } 5L\ \ \ .
$$
It follows that
$$
B \cap \limsup_{i\to\infty}E_i \cap \supp\mu \ \subset \ \bigcup_{L\in {\cal F} } L \cap \supp\mu
\ \subset \ \bigcup_{L \in {\cal G} } 5L \cap \supp\mu\ \ \ .
$$
Hence,
\begin{eqnarray*}
\mu \left(\bigcup_{L \in {\cal G} } 5L\right)  \
\geq \ \mu \big(B \cap
\limsup_{i\to\infty}E_i \big)  \stackrel{\eqref{e:011}}{\ = \ }
  \mu(B) \ .
\end{eqnarray*}

\noindent However, since ${\cal G} $ is a disjoint collection of
balls centred in $\supp\mu$, we have that
\begin{eqnarray*}
\mu \left(\bigcup_{L \in {\cal G} } 5L\right) \ \leq \ \sum_{L \in {\cal G} } \mu \left( 5L\right)
\stackrel{\eqref{doub}}{\ \leq \ }     \ \lambda^3   \  \sum_{L \in {\cal G} } \mu \left( L\right) \ =\     \ \lambda^3 \  \
\mu \left(\bigcup^\circ_{L \in {\cal G} } L\right)\
\ .
\end{eqnarray*}
Thus,
\begin{equation}
\frac{1}{\lambda^3} \ \mu(B)  \; \le \; \mu \left(\bigcup^\circ_{L\in
{\cal G} } L\right)\ = \sum_{L \in
{\cal G} }  \mu \left(L\right) \ \leq
\ \mu(B) \ . \label{kgb001+}
\end{equation}
The sum in \eqref{kgb001} is convergent and therefore
there exists a finite sub-collection $\kgb\subset\cG$ for which
\begin{equation}
\sum_{L \in \kgb}\mu \left( L\right)\ = \
\mu \left(\bigcup^\circ_{L\in \kgb } L\right) \ \ge \ \frac12 \;\frac{1}{\lambda^3}
 \; \ \mu(B) \ .
\label{kgb002+}
\end{equation}
In view of (\ref{kgb001+}) and
(\ref{kgb002+}) the collection $\kgb$
satisfies \eqref{e:014} with $\kappa=\frac{1}{2\lambda^3}$. As already mentioned above, to complete the proof we simply replicate Steps~3~\&~4 in the proof of Proposition~\ref{sv2021}. In remains to note that within the sequence of balls arising at Step~4 there may (and most likely will) be finite disjoint collections of balls arising from the same set $E_i$. These can be grouped together in an obvious manner to form the sequence $(L_{s,B})_{s\in\N}$ as required in the statement of Theorem~\ref{opensets}.

\section{Examples of applications}

In this section we will provide two basic examples showing the conclusions of our results in action.
We wish to emphasize that the applications we discuss in this section are not new -- they have been chosen to demonstrate the key principles in a relatively simple format. New interesting recent applications can be found, for example, in \cite{Daviaud}. We start with an explicit application utilising the power of trimming within a proof of Khintchine's theorem. The proof we provide is not entirely new but, to the best of our knowledge, is simpler, due to some technical simplifications, than the existing published proofs. At the same time it leads to a slightly stronger statement than the standard one.

\subsection{The power of trimming:  Khintchine's theorem   \label{poi}}

Let $\psi:\N\to(0,+\infty)$ be a real,  positive   function.  For $q \in \N$, let
$$
E_q=E_q(\psi):= \{x\in[0,1]:\|qx\|<\psi(q)\},
$$
where $\|\cdot\|$ denotes the distance to the nearest integer, and in turn consider the $\limsup$ set
$$
W(\psi):=\limsup_{q\to\infty} E_q\,.
$$
For obvious reasons, $W(\psi)$ is usually referred to as the set of
\emph{$\psi$-well approximable numbers}.
Khintchine's fundamental theorem \cite{Kh24} in the theory of metric Diophantine approximation dates back to 1924  and it provides an elegant criterion for the `size' of $W(\psi)$ expressed in terms of one-dimensional Lebesgue measure $\lambda$.

\begin{KT} \label{KT}
Let $\psi:\N\to(0,+\infty)$ be such that $\psi(q)/q$ is monotonically decreasing. Then
  $$
 \lambda(W(\psi))=\left\{\begin{array}{ll}
 0 & \text{if  \ }\sum_{q=1}^\infty\psi(q)<\infty\,,\\[2ex]
 1 & \text{if  \ }\sum_{q=1}^\infty\psi(q)=\infty\,.
                         \end{array}
 \right.
  $$
\end{KT}

\bigskip

\begin{rem}
The above statement of Khintchine's Theorem  is in fact a slighter stronger form of the standard modern version  \cite{BDV06} in  which the $\psi(q)$ is assumed to be  monotonically decreasing.
\end{rem}

The convergence part of Khintchine's theorem is an immediate consequence of Lemma~CBC on noting that $\lambda(E_q)\le2\psi(q)$. It does not require the monotonicity assumption or indeed any other additional assumptions. In turn, the modern-days proofs of the divergence part of Khintchine's theorem exploits the principles set out in the main theorems of this paper.
For $q\in\N$ and $p\in\Z$ with $0\le p\le q$ define the balls (intervals) in $\R$
$$
E_{q,p}=\left\{x\in[0,1]:\left|x-\frac pq\right|<\frac{\psi(q)}{q}\right\}\,.
$$
Clearly, $E_q=\bigcup_{p=0}^q E_{q,p}$ and so $W(\psi)$ is the limsup set of the intervals $E_{q,p}$. In view of Cassels' zero-one law \cite{Cassels-50:MR0036787}, $\lambda(W(\psi))=1$ if and only if $\lambda(W(\psi))>0$. In turn, by Theorem~\ref{iffballspositive},
$\lambda(W(\psi))>0$ if and only if there exists a subsequence $(L_i)_i$ of $(E_{q,p})_{q\in\N,0\le p\le q}$ satisfying \eqref{vb30} and \eqref{vb31}.  The upshot of this is that establishing Khintchine's theorem boils down to finding the ``trimmed'' subsequence $(L_i)_{i\in \N}$. This can be done in several ways but probably the easiest is to impose the explicit  condition that the rational fractions $p/q$ under consideration are reduced; that is
$$
(L_i)_{i \in \N}   \, := \,  (E_{q,p})_{q\in\N,\;1\le p\le q,\; \gcd(p,q)=1}\;.
$$
For completeness we present an argument showing the validity of \eqref{vb30} and \eqref{vb31} for this ``trimmed'' subsequence, a version of which can be found in \cite[\S I.3]{Sprindzuk}.

To verify \eqref{vb30}, we start by observing that there are exactly $\varphi(q)$ (the Euler function) positive integers $p\le q$ such that $\gcd(p,q)=1$, and therefore we have that %\sv{why disjoint - disjointness is not needed as this is not measure of the union}
\begin{equation}\label{vb001}
\sum_{\substack{p=0\\[0.5ex] \gcd(p,q)=1}}^q\lambda(E_{q,p})\;=\;\frac{2\varphi(q)\psi(q)}{q}.
\end{equation}
We shall use the following well-known partial summation formula:
$$
\sum_{q=1}^Ta_qb_q=\sum_{q=1}^T(a_q-a_{q+1})(b_1+\dots+b_q)+a_{T+1}(b_1+\dots+b_T)\,,
$$
where $(a_q)_{q\in\N}$ and $(b_q)_{q\in\N}$ are any two sequences or real numbers, and the following well known asymptotics for the average order of the Euler's function:
$$
\sum_{q=1}^{T}\varphi(q)\sim\frac{3}{\pi^2}T^2~~~\text{as }~T\to\infty\,.
$$
Let $0<C_1<3/\pi^2$. Then, using the fact that $\psi(q)/q$ is decreasing, \eqref{vb001} and the trivial estimate $1+\dots+q\le q^2$, by the partial summation formula with $a_q=2\psi(q)/q$, $b_q=\varphi(q)$, we have that for sufficiently large $T$
\begin{eqnarray*}
\sum_{q=1}^{T}\sum_{\substack{p=0\\[0.3ex] \gcd(p,q)=1}}^q\lambda(E_{q,p}) & \ge  &
\sum_{q=1}^T(a_q-a_{q+1})C_1q^2+a_{T+1}C_1T^2  \\
& \ge &
C_1\left(\sum_{q=1}^T(a_q-a_{q+1})(1+\dots+q)+a_{T+1}(1+\dots+T)\right)\,.
\end{eqnarray*}
And again by the partial summation formula, this time with $a_q=2\psi(q)/q$ and $b_q=q$, we get that the above equals $C_1\sum_{q=1}^Tqa_q=2C_1\sum_{q=1}^T\psi(q)$. Hence,
\begin{equation}\label{vb3+}
\sum_{q=1}^{T}\sum_{\substack{p=0\\[0.3ex] \gcd(p,q)=1}}^q\lambda(E_{q,p})  \ \ge  \ 2C_1\sum_{q=1}^T\psi(q)
\end{equation}
for sufficiently large $T$.
In particular, this implies \eqref{vb30}.

To verify \eqref{vb31}, first observe that if $q<m$, $1\le p\le q$ and $1\le k\le m$
then $\lambda(E_{q,p}\cap E_{m,k})\le \lambda(E_{m,k})\le2\psi(m)/m$. Then, for fixed $q<m$ we get that
\begin{equation}\label{c1}
\sum_{\substack{p=0\\[0.3ex] \gcd(p,q)=1}}^q\;\;
\sum_{\substack{k=0\\[0.3ex] \gcd(k,m)=1}}^m
\lambda(E_{q,p}\cap E_{m,k})\;\le\; \frac{2\psi(m)}{m}\times \#\big\{(k,p):E_{q,p}\cap E_{m,k}\neq\varnothing\big\}\,.
\end{equation}
Further, if $x\in E_{q,p}\cap E_{m,k}$ then trivially $|qx-p|<\psi(q)$ and $|mx-k|<\psi(m)$,
whence
$$
|pm-qk|\le m\psi(q)+q\psi(m)\le 2m\psi(q)\,.
$$
Also, since the fractions $p/q$ and $k/m$ are reduced and different (for we assumed that $m> q$) we must have that
$|pm-qk|\ge1$. Thus the number of $(p,k)$ in the right hand side of \eqref{c1} is less than or equal to
the number of integer points $(p,k)$ satisfying
\begin{equation}\label{vb002}
1\le p\le q, ~1\le k\le m, ~
1\le |pm-qk|\le 2m\psi(q)\,.
\end{equation}
If all such points $(p,k)$ lie on a line, then from the last inequality of \eqref{vb002} we immediately get that their number is $\le 4m\psi(q)$. Otherwise, assuming such points exit, the set of these points has rank 2 and, by \eqref{vb002}, lies in the convex body given by
$
|p|\le q, ~|k|\le m, ~
|pm-qk|\le 2m\psi(q)
$
which has volume $\le 16m\psi(q)$. In this case, the number of such points is bounded by $32m\psi(q)+2\le 36m\psi(q)$ as a consequence of Blichfeldt's theorem \cite{Blichfeldt-1921}. Either way, the right hand side of \eqref{c1} is bounded by
$72\psi(m)\psi(q)$. Clearly, the same holds when $q>m$. Therefore, in view of the divergence sum condition
\begin{eqnarray*}
\sum_{q=1}^T\;\sum_{m=1}^T
\sum_{\substack{p=0\\[0.3ex] \gcd(p,q)=1}}^q\;\;
\sum_{\substack{k=0\\[0.3ex] \gcd(k,m)=1}}^m
\lambda(E_{q,p}\cap E_{m,k})
& \le  & 72\left(\sum_{q=1}^T\psi(q)\right)^2+2\sum_{q=1}^T\psi(q) \\ & \le & 73\left(\sum_{q=1}^T\psi(q)\right)^2
\end{eqnarray*}
for sufficiently large $T$. Together with \eqref{vb3+} this verifies \eqref{vb31} with $C=73/(4C_1^2)$.

\medskip

\begin{rem}
The question regarding the relevance of monotonicity in Khintchine's theorem remained a prominent open problem in probabilistic number theory for nearly 80 years. Indeed, in 1941 Duffin $\&$ Schaeffer showed that the monotonicity could not be removed (by providing a counterexample) and they formulated an alternative statement. This attracted much work (by Erd\"os, Vaaler, Pollington, Vaughan and Harman amongst others) and was eventually proved by Koukoulopoulos  $\&$ Maynard \cite{maynard}. All these works used trimming as the basis for their approaches very much in line with the outline above. Of course, the process and implementation of trimming are significantly more sophisticated.
\end{rem}

\medskip

\begin{rem}
The above example makes use of the power of trimming within the context of Theorem~\ref{iffballspositive}, a statement dealing with positive measure. In turn, the ``ubiquity'' technique \cite{BDV06} represents an example of the power of trimming within the context of Theorem~\ref{iffballs},  a statement dealing with full measure.  In short, the theory of ubiquitous systems provides a general framework for deducing full measure statements for a large class of $\limsup$  sets and in view of Theorem~\ref{iffballs},  it is not at all surprising  that ``trimming'' plays a central role when developing the theory.
\end{rem}

\medskip

Returning to Theorem~K, note that the convergence part implies that
$$
\lambda\big( W(\tau) \big) = 0 \qquad {\rm for \ any } \quad \tau > 1 \, ,
$$
where  for any  $\tau > 0 $ we  write $ W(\tau)$ for $ W(\psi:q \to q^{-\tau})$.
The set $W(\tau)$ is usually referred to as the set of
\emph{$\tau$-well approximable numbers\index{tau-well approximable numbers}}. The upshot of the above is that for
any $ \tau > 1$, the set of $\tau$-well approximable numbers is of
measure zero and we cannot obtain any further information regarding the
`size' of $W(\tau)$ in terms of Lebesgue measure --- it is always
zero. Intuitively, the `size' of $W(\tau)$ should decrease as $\tau$ increases.  In short, we require a more delicate notion of `size'
than simply Lebesgue measure. The appropriate notion of `size' best
suited for describing the finer measure theoretic structures of
$W(\tau)$ and indeed $W(\psi )$ is that of Hausdorff measures.

Let $(\Omega, d) $ be a  metric space and let $X$ be a subset of $\Omega$.   For $\rho
> 0$, a countable collection $ \left\{B_{i} \right\} $ of  balls in  $\Omega$ of radius $r_i \le   \rho $ for each
$i$ such that $X \subset \bigcup_{i} B_{i} $ is called a $ \rho
$-cover for $X$.  Let $s$ be a non-negative number and define $$
 {\cal H}^{s}_{ \rho } (X)
  \; := \; \inf \big\{ \sum_{i} r_i^s
\ :   \{ B_{i} \}  {\rm \  is\ a\  } \rho {\rm -cover\  of\ } X
\big\} \; , $$ where the infimum is taken over all possible $
\rho $-covers of $X$. The {\it s-dimensional Hausdorff measure}
${\cal H}^{s} (X)$ of $X$ is defined by $$ {\cal H}^{s} (X) :=
\lim_{ \rho \rightarrow 0} {\cal H}^{s}_{ \rho } (X) = \sup_{ \rho
> 0} {\cal H}^{s}_{ \rho } (X)
$$ \noindent and the {\it Hausdorff dimension} dim $X$ of $X$ by
$$ \dim_{\rm H} \, X := \inf \left\{ s\ge0 : {\cal H}^{s} (X) =0 \right\}  \, . $$
It is worth emphasizing that when $s$ is a
positive integer, then $\cH^s$ is a constant multiple of Lebesgue measure in $\R^s$.  Indeed,  when $s=1$ $\cH^s$ is $\tfrac12\lambda$. In particular,  $  \cH^1 ([0,1]) = \tfrac12\lambda([0,1])$ and it follows from the definition of Hausdorff dimension that
\begin{equation*}
    \mathcal{H}^s\big([0,1]\big) = \begin{cases}
      0 &{\rm if} \;\;\;   s> 1 \,, \\
      \infty  & {\rm if} \;\;\; s<1 \,.
    \end{cases}
  \end{equation*}
For further details concerning Hausdorff measure and dimension see \cite{falc, jh, mat}.

\medskip

The following statement  is a Hausdorff measure
analogue of Khintchine's Theorem. It provides an elegant
criterion for the `size' of the set $W(\psi) $ expressed in terms of
 the measure  $\cH^s$.   The convergent part is an immediate consequence of the natural generalization of Lemma~CBC  to Hausdorff measures (see for example \cite[Lemma~3.10]{BDod}).   As with  Khintchine's theorem,  the main substance is very much  the divergence part.

\begin{KJ}
  \label{fulllmhm}
  Let $\psi:\N\to(0,+\infty)$ be such that $\psi(q)/q$ is monotonically decreasing and let $s \in (0,1]$. Then
  \begin{equation*}
    \mathcal{H}^s\big(W(\psi)\big) = \begin{cases}
      0 &{\rm if} \;\;\;  \sum_{q=1}^{\infty}q^{1-s} \psi^s(q)<\infty\,, \\[2ex]
      \mathcal{H}^s(\I) & {\rm if} \;\;\; \sum_{q=1}^{\infty}q^{1-s} \psi^s(q)=\infty\,.
    \end{cases}
  \end{equation*}
\end{KJ}

\noindent Recall, that $\cH^1= \tfrac12\lambda$ and so when $s=1$ the above reduces to Theorem~K.  When $ s < 1 $,  the above Hausdorff measure statement is essentially due to Jarn\'{\i}k and dates back to 1931.
Note that in this case $ \cH^s ([0,1]) = \infty $ and Jarn\'{\i}k Theorem (i.e. Theorem~K-J with $s<1$) implies that
$$
\dim W(\tau)    \, = \,  \frac{2}{1 + \tau} \qquad (\tau  \geq  1) \, .
$$
Hence the
the `size' of $W(\tau)$ decreases as $\tau$ increases which is inline with our intuition.   For further details and a gentle introduction to the theory of metric Diophantine approximation see \cite{durham}.  %HarmanMNT, Sprindzuk}

\medskip

The second application of our results constitutes the key element of the so-called Mass Transference Principle which enable us to deduce Theorem~K-J from Theorem~K.   At first glance this seems rather odd since Hausdorff measures are regarded as a natural refinement of Lebesgue measure.

\subsection{The power of full measure: Mass Transference Principle}

The second key example exhibits the power of full measure.
To set the scene, let $(\Omega, \cA, \mu, d) $ be a locally compact metric measure space equipped with a Borel regular probability measure $\mu$. Without loss of generality we will assume that $\Omega $ is the support of $\mu$.  With this in mind, suppose there
exist constants $ \delta > 0$, $0<a\le 1\le b<\infty$ and $r_0 > 0$ such that
\begin{equation} \label{MTPmeasure}
 a \, r ^{\delta}  \ \leq  \  \mu(B)  \ \leq  \   b \, r
^{\delta}
\end{equation}
for any ball $B=B(x,r)$ with $x\in \Omega$ and radius $r\le r_0$. Such a measure is said to be \emph{Ahlfors $\delta$-regular}. It is well know that if $\Omega$ supports an Ahlfors $\delta$-regular measure $\mu$, then   $\dim_{\rm H} \Omega = \delta$ and moreover
that $ \mu$ is strongly equivalent to $\delta$-dimensional Hausdorff measure ${\cal H}^\delta$ -- see \cite{falc,jh,mat} for
details. The latter simply means that there exists a constant $C\ge 1$ such that
for every $\mu$-measurable subset $ E $ of $\Omega$
\[
{ C^{-1}} {\cal H}^\delta(E) \leq \mu(E) \leq C {\cal H}^\delta(E)
\] and so  \eqref{MTPmeasure} is equally valid with $\mu$ replaced by  $\cH^\delta $.  Also note that it is easily verified that a $\delta$-Ahlfors regular measure is a doubling measure.  Finally, throughout this section, given $s > 0$  and a ball $B$ we define
the scaled ball
$$
B^s:=B\big(x,r^{\frac{s}{\delta}}\big)\,  .
$$
Note, by definition $B^{\delta}=B$ and  if $r<1$ and $s < \delta $ then $B^s$ is a scaled up  ball.

\medskip

Let $\{B_i\}_{i\in\N}$ be a sequence of balls in $\Omega$ with radius
$\diam(B_i)\to 0$ as $i\to\infty$ and suppose that
$$
\sum_{i=1}^{\infty} \mu(B_i) <  \infty  \, .
$$
In view of Lemma~CBC, it follows that
$$ \mu \big( \limsup_{i\to\infty}B_i
  \big)    = 0 =  \cH^\delta  \big( \limsup_{i\to\infty}B_i
  \big)   \, .     $$
However, now suppose there exists some $s > 0   $ such that the $\limsup$ set associated with the  scaled up balls $B_i^s$ has full measure; that is
 $$\mu \big( \limsup_{i\to\infty}B_i^s
  \big)    = 1 =  \cH^\delta  \big( \limsup_{i\to\infty}B_i^s
  \big)   \, .     $$
It turns out that knowing such a  full measure statement for the ``scaled up'' balls  enables us to deduce an analogous statement for the original balls.  Indeed, the following  Mass Transference Principle \cite[Theorem~3]{MTP} allows us to transfer
$\cH^\delta$-measure theoretic statements for $\limsup$ subsets of
$\Omega$ to general $\cH^s$-measure theoretic statements.

\begin{MTP}\label{MTPB}
Let $(\Omega, \cA, \mu, d) $ be a locally compact metric measure space equipped with a Borel regular $\delta$-Ahlfors regular probability measure $\mu$ supported on $\Omega$.   Let $\{B_i\}_{i \in \N}$ a sequence of balls in $\Omega$ with radius $r(B_i)\to 0$ as $i\to\infty$.
		Let $s \ge  0 $ and suppose that
$$
\mathcal{H}^\delta \big( \limsup_{i\to\infty}B_i^s\big)=\mathcal{H}^\delta(\Omega)\qquad\text{or equivalently}\qquad
\mu\big( \limsup_{i\to\infty}B_i^s\big)=\mu(\Omega).
$$
		Then,		$$\mathcal{H}^s\big(\limsup_{i\to\infty}B_i\big)=\mathcal{H}^s(\Omega).$$
	\end{MTP}
	
\bigskip

\begin{rem} Note that by the definition of Hausdorff dimension,  Theorem~MTP implies that
$ \dim_{\rm H} \big(\limsup_{n\to\infty}B_n   \big)  \ge   s \, , $
and moreover that
$\mathcal{H}^s(\limsup_{n\to\infty}B_n) = \infty    $ if $s < \delta
$.
\end{rem}

\medskip

With reference to Proposition~\ref{sv2021}, the key towards establishing the Mass Transference Principle is to make use of the fact that the full measure statement (A) implies the existence of the finite
sub-collection $\kgb$ of balls satisfying (C). In \cite{MTP}, this implication is explicitly the subject of Section 4.  In short it provides deep information regarding the local distribution of the centres of the  balls under consideration.  This is very much at the heart of the ``optimal'' Cantor construction carried out in \cite[Section 5]{MTP} that enables one to show that $\mathcal{H}^s(\limsup_{n\to\infty}B_n) = \infty $  ($=  \mathcal{H}^s(\Omega)$)  if $s < \delta$.  The Cantor construction itself is more technical rather than innovative  -- the existence of the collection $\kgb$ is the crux!

\vspace*{2ex}

\begin{rem} There have been a steady series of works \cite{AB, ABer,  MTPslice, KR, P, WW, WWW, Zh} that  extend the Mass Transference Principle in numerous  directions,  such as to systems of linear forms,  iterated function schemes and large intersection sets. For an overview of the first ten years after Theorem~MTP,  we refer the reader to the review article \cite{AT}.    The more recent work of Wang $\&$ Wu \cite{WW} is particularly notable in that it deals with  $\limsup$ sets defined via rectangles rather than simply balls.  It is well worth stressing that all the above cited variants of Theorem~MTP  have at their heart a common feature. In one form or another,  they all exploit the fact that any full measure statement such as (A) in Proposition~\ref{sv2021} implies the existence of the finite sub-collection $\kgb$ of balls satisfying (C).
\end{rem}

We bring this section to a close by using Theorem~MTP to show that within the world of classical  metric Diophantine approximation as described in \S\ref{poi},  the Lebesgue theory of $\limsup$ sets underpins the general Hausdorff theory.  This is rather surprising since the latter theory is  regarded to be a subtle refinement of the former.

  The claim is that in view of the Mass Transference Principle
we have  that
\begin{center} Khintchine's Theorem
  $ \hspace{9mm} \Longrightarrow \hspace{9mm} $ Jarn\'{\i}k's Theorem  ;
  %\\ $ \hspace{15mm} $
%   (Theorem K) $ \hspace{30mm} $ (Theorem K-J with $ s < 1 $)
    \end{center}
    i.e., Theorem~K (which is of course Theorem K-J with $s=1$) implies Theorem~K-J for all  $ s \in (0,1)$.   First of all let us dispose of the case that
$\psi(q)/q \nrightarrow 0 $ as $q \to \infty$. Then trivially,
$W(\psi)= \I$ and the result is obvious. Without loss of generality,
assume that $\psi(q)/q \to 0 $ as $q \to \infty$. With respect to the
Mass Transference Principle, let $\Omega = \I$, $ d $ be the supremum
norm, $\delta = 1$ and  $s \in (0,1)$.  We are given that $\psi(q)/q$ is monotonically decreasing and
that $ \sum q^{1-s} \psi(q)^s = \infty $. Let
$ \theta(q) := q^{1-s} \psi(q)^s$. Then it follows that $\theta(q)/q$ is  monotonically decreasing and $ \sum \theta(q) = \infty $.  Thus, Khintchine's Theorem
implies that $ \cH^1( W(\theta)) = \cH^1( \I)$. It now follows via the Mass Transference Principle
that $ \cH^s(W(\psi)) = \cH^s(\I)= \infty $ and this completes the
proof of the divergence part of Jarn\'{\i}k's Theorem -- the main substance of Theorem~K-J. As mentioned in \S\ref{poi}, the convergence part of Theorem~K-J is a straight forward consequence of Lemma~CBC for  Hausdorff measures \cite[Lemma~3.10]{BDod}.

\vspace*{0.5ex}

\noindent{\bf Acknowledgements.} This work was partly supported by EPSRC Programme grant EP/J018260/1.

%\vspace*{1ex}

\vspace{0ex}

{\small
\noindent Victor Beresnevich:\\
 Department of Mathematics,
 University of York,\\
 Heslington, York, YO10 5DD,
 England\\
 E-mail: {\tt victor.beresnevich@york.ac.uk}

\bigskip

\noindent Sanju Velani:\\
 Department of Mathematics,
 University of York,\\
 Heslington, York, YO10 5DD,
 England\\
 E-mail: {\tt sanju.velani@york.ac.uk}

}

\end{document}